\documentclass[10pt]{article}
\usepackage[bb=fourier,cal=pxtx]{mathalfa}
\usepackage{amssymb}
\usepackage{mathtools}		
\usepackage{mathabx}		

\usepackage{amsthm}
\usepackage{thmtools}
\usepackage{tikz}			
\usepackage[colorlinks=true]{hyperref}	
\usepackage[font=footnotesize]{caption}
\usepackage[capitalise]{cleveref} 

\usepackage{xcolor}
\hypersetup{
    colorlinks,
    linkcolor={red!50!black},
    citecolor={blue!50!black},
    urlcolor={blue!80!black}
}

\makeatletter

\def\MRbibitem{\@ifnextchar[\my@lbibitem\my@bibitem}

\def\mybiblabel#1#2{\@biblabel{{\hyperref{http://www.ams.org/mathscinet-getitem?mr=#1}{}{}{#2}}}}

\def\myhyperanchor#1{\Hy@raisedlink{\hyper@anchorstart{cite.#1}\hyper@anchorend}}

\def\my@lbibitem[#1]#2#3#4\par{%
  \item[\mybiblabel{#2}{#1}\myhyperanchor{#3}\hfill]#4%
  \@ifundefined{ifbackrefparscan}{}{\BR@backref{#3}}%
  \if@filesw{\let\protect\noexpand\immediate
    \write\@auxout{\string\bibcite{#3}{#1}}}\fi\ignorespaces%
}

\def\my@bibitem#1#2#3\par{%
  \refstepcounter\@listctr
  \item[\mybiblabel{#1}{\the\value\@listctr}\myhyperanchor{#2}\hfill]#3
  \@ifundefined{ifbackrefparscan}{}{\BR@backref{#2}}%
  \if@filesw\immediate\write\@auxout
    {\string\bibcite{#2}{\the\value\@listctr}}\fi\ignorespaces%
}

\makeatother

\declaretheorem[numberwithin=section]{theorem}

\declaretheorem[sibling=theorem]{proposition}

\declaretheorem[name=Acknowledgement, style=definition, numbered=no]{ack}

\hypersetup{bookmarksdepth = 3} 
\numberwithin{equation}{section}     

\newcommand{\biangle}[1]{\mathopen{\langle {\kern -1.5pt} \langle}{#1} \mathclose{\rangle {\kern -1.5pt} \rangle}}

\newcommand{\cE}{\mathcal{E}}\newcommand{\cF}{\mathcal{F}}

\newcommand{\cM}{\mathcal{M}}

\newcommand{\cS}{\mathcal{S}}
\newcommand{\cW}{\mathcal{W}}

\newcommand{\C}{\mathbb{C}}
\newcommand{\R}{\mathbb{R}}
\newcommand{\Z}{\mathbb{Z}}
\newcommand{\N}{\mathbb{N}}
\renewcommand{\P}{\mathbb{P}}

\newcommand{\GL}{\mathrm{GL}}
\newcommand{\SL}{\mathrm{SL}}
\newcommand{\Mat}{\mathrm{Mat}}

\newcommand{\st}{\;\mathord{;}\;}

\DeclareMathOperator{\supp}{supp}

\newcommand*\circled[1]{\tikz[baseline=(char.base)]{
    \node[shape=circle,draw,inner sep=1pt] (char) {\footnotesize{#1}};}}

\newcommand{\dd}{\, \mathrm{d}}   
\DeclareMathOperator{\midpoint}{mid}
\newcommand{\conv}{\operatorname{\mathrm{co}}}
\newcommand{\ext}{\operatorname{\mathrm{ext}}}
\newcommand{\s}{\mathbf{s}} 
\newcommand{\fa}{\mathfrak{a}}

\newcommand{\arxiv}[1]{\href{http://arxiv.org/abs/#1}{arXiv:{#1}}}

\renewcommand{\epsilon}{\varepsilon}
\renewcommand{\phi}{\varphi}

\renewcommand{\emptyset}{\varnothing}

\makeatletter
\newcommand{\subjclass}[2][1991]{%
  \let\@oldtitle\@title%
  \gdef\@title{\@oldtitle\footnotetext{#1 \emph{Mathematics subject classification.} #2}}%
}
\newcommand{\keywords}[1]{%
  \let\@@oldtitle\@title%
  \gdef\@title{\@@oldtitle\footnotetext{\emph{Key words and phrases.} #1.}}%
}
\newcommand{\ackno}[1]{%
  \let\@@oldtitle\@title%
  \gdef\@title{\@@oldtitle\footnotetext{\emph{Acknowledgements.} #1.}}%
}
\makeatother

\begin{document}

\title{Ergodic optimization of Birkhoff averages and Lyapunov exponents}
\date{February 2, 2017}
\author{Jairo Bochi \\ \small Facultad de Matem\'aticas, Pontificia Universidad Cat\'olica de Chile}


\ackno{Partially supported by project Fondecyt 1140202}

\maketitle

\addcontentsline{toc}{section}{Introduction}

In this paper $(X,T)$ denotes a topological dynamical system, that is, $X$ is a compact metric space and $T \colon X \to X$ is a continuous map.
Often we will impose additional conditions, but broadly speaking the dynamics that interest us the most are those that are sufficiently ``chaotic'', and in particular have many invariant probability measures.

Our subject is \emph{ergodic optimization} in a broad sense, meaning the study of extremal values of asymptotic dynamical quantities, and of the orbits or invariant measures that attain them. 
More concretely, we will discuss the following topics:
\begin{enumerate}
\item 
maximization or minimization of the ergodic averages of a real-valued function; 
\item 
optimization of the ergodic averages of a vectorial function, meaning that we are interested in the extrema of the ergodic averages of a function taking values in some euclidian space $\R^d$; 
\item 
maximization or minimization of the top Lyapunov exponent of a linear cocycle over $(X,T)$, or more generally, of the asymptotic average of a subadditive sequence of functions; 
\item
optimization of the whole vector of Lyapunov exponents of a linear cocycle.
\end{enumerate}
Unsurprisingly, many basic results and natural questions that arise in these topics are parallel. 
The aim of this paper is to provide an unified point of view, hoping that it will attract more attention to the many open problems and potential applications of the subject.
The setting should also be convenient for the study of problems where one cares about classes of invariant measures that are not necessarily optimizing. 

Disclaimer: This mandatorily short article is not a survey.
We will not try to catalog the large corpus of papers that fit into the subject of ergodic optimization.
We will neither provide a historical perspective of the development of these ideas, nor explore connections with fields such as Lagrangian Mechanics, Thermodynamical Formalism, Multifractal Analysis, and Control Theory.\footnote{I recommend Jenkinson's new survey \cite{Jenk_new}, which appeared shortly after the conclusion of this paper.}

\section{Optimization of Birkhoff averages} \label{s.Birkhoff_one_dim}

We denote by $\cM_T$ the set of $T$-invariant probability measures, which is a nonempty convex set and is compact with respect to the weak-star topology. Also let $\cE_T \subseteq \cM_T$ be the subset formed by ergodic measures, which are exactly the extremal points of $\cM_T$.

Let $f \colon X \to \R$ be a continuous function.
We use the following notation for Birkhoff sums:
$$
f^{(n)} \coloneqq f + f\circ T + \cdots +  f\circ T^{n-1} \, .
$$
By Birkhoff theorem, for every $\mu \in \cM_T$ and $\mu$-almost every $x \in X$, the asymptotic average $\lim_{n \to \infty }\tfrac{1}{n} f^{(n)}(x)$ is well-defined. The infimum and the supremum of all those averages will be denoted by $\alpha(f)$ and $\beta(f)$, respectively; we call these numbers the \emph{minimal} and \emph{maximal ergodic averages} of $f$.
Since $\alpha(f) = - \beta(-f)$, let us focus the discussion on the quantity $\beta$.
It can also be characterized as:
\begin{equation}\label{e.integral_beta}
\beta(f) = \sup_{\mu \in \cM_T} \int f \dd\mu  \, . 
\end{equation}
Compactness of $\cM_T$ implies the following \emph{attainability property}:
there exists at least one measure $\mu \in \cM_T$ for which $\int f \dd \mu = \beta(f)$;
such measures will be called \emph{maximizing measures}.

Another characterization is given by the following \emph{enveloping property}:
\begin{equation}\label{e.envelope_beta}
\beta(f) =
\inf_{n \ge 1} \frac{1}{n} \sup_{x \in X} f^{(n)}(x) \, , \quad \text{and the $\inf$ is also a $\lim$.}
\end{equation}
Actually, upper semicontinuity of $f$ suffices for these characterizations: see \cite{Jenk.survey}.

Recall that a function of the form $h \circ T - h$ with $h$ continuous is called a \emph{coboundary} (or \emph{$C^0$-coboundary}).
Two functions that differ by a coboundary are called \emph{cohomologous}, and have the same maximal ergodic average $\beta$.
Actually, $\beta(f)$ can be characterized as a minimax over the cohomology class of $f$:
\begin{equation}\label{e.dual}
\beta(f) = \inf_{h \in C^0(X)} \sup_{x \in X} (f + h \circ T - h) \, .
\end{equation}
Following the terminology in linear programming, this is called the \emph{dual characterization} of $\beta(f)$: see \cite{Radu}.
Formula \eqref{e.dual} was discovered independently several times; it is Lemma~1.3 in the paper \cite{FuKi}, where it is proved using Hahn--Banach theorem. Let us reproduce the more direct proof from \cite{Conze-G}.
The finite-time average $\tfrac{1}{n}f^{(n)}$ is cohomologous to $f$; indeed it equals $f + h \circ T - h$ for $h \coloneqq \tfrac{1}{n}\sum_{i=1}^n f^{(i)}$. Using the enveloping property \eqref{e.envelope_beta} we obtain the $\ge$ inequality in the dual formula \eqref{e.dual}. The reverse inequality is trivial.

Given $f$, a natural question arises: is the infimum in \eqref{e.dual} attained?
Well, if $h \in C^0(X)$ attains the infimum, then the inequality $f + h \circ T - h \le \beta(f)$ holds everywhere on $X$. Consider the closed set $K$ where equality holds.
This set is nonempty; indeed by integrating the inequality we see that a measure $\mu \in \cM_T$ is maximizing if and only if
$\mu(K)=1$, that is, if $\supp \mu \subseteq K$.
Any closed set with this property is called a \emph{maximizing set}.
So another question is whether the existence of such sets is guaranteed. 
Let us postpone answers a bit.

Every continuous function $f$ can be seen as a linear functional on the vector space of signed Borel measures on $X$, and conversely. 
The quantity $\beta(f)$ is the maximum that this linear functional attains in the compact convex set $\cM_T$, and so its computation is a problem of infinite-dimensional linear programming. 

Since every ergodic $\mu$ is an extremal point of $\cM_T$, it is maximizing for some $f$. Furthermore, since $\cM_T$ is a simplex (by uniqueness of ergodic decompositions), 
every ergodic $\mu$ is the unique maximizing measure for at least one function $f \in C^0(X)$ (see \cite{Jenk.unique} for the precise arguments). Conversely, if $f$ has a unique maximizing measure $\mu$ then $\mu$ must be ergodic (because the ergodic decomposition of a maximizing measure is formed by maximizing measures).

Uniqueness of the maximizing measure is a \emph{(topologically) generic} property, i.e., it holds for every function in a dense $G_\delta$ subset of $C^0(X)$.
Furthermore, the same is true if $C^0(X)$ is replaced by any Baire vector space $\cF$ of functions that embeds continuously and densely in $C^0(X)$: see \cite[Thrm.~3.2]{Jenk.survey}. 

The properties of maximizing measures of generic functions in $\cF$ may be very different according to the space under consideration. Consider $\cF = C^0(X)$ first. Suppose that $T$ is sufficiently hyperbolic (more precisely, $T$ satisfies Bowen's specification property); to avoid trivialities also assume that $X$ is a perfect set. Then, as shown by Morris \cite{Morris.Baire}, the unique maximizing measure of a generic function $f \in C^0(X)$ satisfies any chosen generic property in the space of measures $\cM_T(X)$; in particular maximizing measures  generically have zero entropy and full support. 
Note that if a function $f$ admits a maximizing set $K$
and has a maximizing measure of full support, then necessarily $K=X$ and therefore all probability measures have the same integral. 
Since the latter property is obviously non-generic (our $T$'s are not uniquely ergodic), we conclude that generic continuous functions $f$ admit \emph{no} maximizing set, and in particular the infimum in the dual formula \eqref{e.dual} is \emph{not} attained.

The situation is radically different for more regular functions. 
A central result in ergodic optimization, found in different forms via various methods by many authors, roughly states that if the dynamics $T$ is sufficiently hyperbolic and the function $f$ is sufficiently regular, then the infimum in the dual formula \eqref{e.dual} is attained.
Such results are called \emph{non-positive Livsic theorems}, \emph{Ma\~n\'e--Conze--Guivarch lemmas}, or \emph{Ma\~n\'e lemmas} for short.
One of the simplest versions is this: if $T \colon X \to X$ is a one-sided subshift of finite type, and $f$ is a $\theta$-H\"older function (assuming $X$ metrized in the usual way), then there exists a $\theta$-H\"older function $h$ such that 
\begin{equation}\label{e.Mane}
f + h \circ T - h \le \beta(f) \, ,
\end{equation}
Similar statements also hold for uniformly expanding maps, Anosov diffeomorphisms and flows, etc. 
Some references are \cite{Conze-G,Savchenko,CLT,Bousch.Walters,Lopes-T,Pol-Sharp,Bousch.amphi,Garibaldi.book}.
The methods of proof are also diverse: some proofs use Thermodynamical Formalism, some use fixed point theorems, and some use bare hands. 
A function $h$ solving the cohomological inequality \eqref{e.Mane} is called a \emph{subaction}.
As negative result, it is shown in \cite{Bousch-Jenk} that the regularity of the subaction $h$ is not always as good as $f$: it may be impossible to find a $C^1$ subaction $h$ even if $T$ and $f$ are~$C^\omega$.
The study of subactions forms a subject by itself: see \cite{Garibaldi.book}.

So it is natural to focus the study on regular functions $f$ and hyperbolic dynamics $T$,
for which the theory is richer.
Yuan and Hunt \cite{Yuan-Hunt} showed that only measures $\mu$ supported on periodic orbits can have the \emph{locking property}, which means that $\mu$ is the unique maximizing measure for some $f$ and also for functions sufficiently close to $f$.
Much more recently, Contreras \cite{Contreras} settled a main open problem and proved that maximizing measures are generically supported on periodic orbits. More precisely, he proved that if $T$ is a expanding map then a generic Lipschitz function has a unique maximizing measure, which is supported on a periodic orbit and has the locking property. 

Contreras' theorem provides some confirmation of the experimental findings of Hunt and Ott published two decades before \cite{Hunt-Ott}. They basically conjectured that for typical chaotic systems $T$ and typical smooth functions $f$, the maximizing measure is supported on a periodic orbit. However, their concept of typicality was a probabilistic one: Hunt and Ott actually conjectured that for typical parameterized families of functions, the Lebesgue measure of the parameters corresponding to maximizing orbits of period $p$ or greater is exponentially small in terms of $p$. This type of conjecture remains open. 

A conceptually clean probabilistic notion of typicality in function spaces was introduced in \cite{HSY} (basically rediscovering \cite{Christ}); it is called \emph{prevalence}. See \cite{Hunt-Kal} for several examples of prevalent properties (not all of them topologically generic) in Dynamical Systems. 
Zhang and the author \cite{B-Zhang} have obtained the following result in the direction of Hunt--Ott conjectures: if $T$ is the one-sided shift on two symbols, and $\cF$ is a space of functions with a very strong modulus of regularity, then every $f$ in a prevalent subset of $\cF$ has a unique maximizing measure, which is supported on a periodic orbit and has the locking property. 
Furthermore, we have obtained a sufficient condition for periodicity in terms of the wavelet coefficients of $f$. There is experimental evidence that this condition is prevalent not only in the space $\cF$ but also on bigger spaces of H\"older functions, but a proof is still missing.

For full shifts (and for other sufficiently hyperbolic dynamics as well), the set of measures supported on periodic orbits is dense in $\cM_T$.
In particular, $\cM_T$ is a \emph{Poulsen simplex}: its set $\cE_T$ of extremal points is dense. 
It seems fanciful to try to form a mental image of such an object, but let us try anyway. 
There are natural ways (see \cite{B-Zhang}) to approximate the Poulsen simplex by a nested sequence $R_1 \subset R_2 \subset \cdots \subset \cM_T$ of (finite-dimensional) polyhedra whose vertices are measures supported on periodic orbits.
Moreover, each polyhedron $R_n$ is a projection of the next one $R_{n+1}$, and the whole simplex $\cM_T$ is the inverse limit of the sequence.
These polyhedra are not simplices: on the contrary, they have a huge number of vertices and their faces are small.
Moreover, the polyhedra are increasingly non-round: the height of $R_{n+1}$ with respect to $R_n$ is small. 
In particular, each interior point of $R_n$ can be well-approximated by a vertex of some $R_m$ with $m > n$; this resembles the Poulsen property.
Furthermore, among the vertices of $R_n$, only a few of them are ``pointy'', and the others are ``blunt''; these pointy vertices are the measures supported on orbits of low period. If we take at random a linear functional on the finite-dimensional span of $R_n$, then the vertex of $R_n$ that attains the maximum is probably a pointy vertex. 
This is a speculative justification for the Hunt--Ott conjectures. 

As mentioned before, not every $\mu \in \cE_T$ can appear as a unique maximizing measure of a \emph{regular} function. So it is natural to ask exactly which measures can appear once the regularity class is prescribed. 
Motivated by a class of examples that we will explain in the next section, 
Jenkinson formulated the following fascinating question \cite[Probl.~3.12]{Jenk.survey}: 
if $T$ is the doubling map on the circle and $f$ is an analytic function with a unique maximizing measure $\mu$, can $\mu$ have positive entropy?

\section{Optimization of vectorial Birkhoff averages} \label{s.Birkhoff_any_dim}

Now consider a continuous vectorial function $f \colon X \to \R^d$. The \emph{rotation set} of $f$ is the set $R(f)$ of all averages $\int f \dd \mu$ where $\mu \in \cM_T$.
This is a compact convex subset of $\R^d$.
Furthermore, by ergodic decomposition, it equals the convex hull of the averages of $f$ with respect to ergodic measures; in symbols:
\begin{equation}\label{e.rot_set_erg}
R(f)
= \conv \left\{ \int f \dd\mu \st \mu \in \cE_T \right\} \, . 
\end{equation}
If $d = 1$ then $R(f) = [\alpha(f), \beta(f)]$, using the notation of the previous \lcnamecref{s.Birkhoff_one_dim}.
The prime type of examples of rotation set, which justifies the terminology, are those when $f$ equals the displacement vector of a map of the $d$-torus homotopic to the identity. Other examples of rotation sets, where the dynamics is actually a (geodesic) flow, are Schwartzman balls \cite{Paternain,Pol-Sharp}.

The measures $\mu \in \cM_T$ for which $\int f \dd\mu$ is an extremal point of $R(f)$ are called \emph{extremal measures}.
Of course, each of these measures is also a maximizing measure for a real-valued function $\langle c, f(\mathord{\cdot}) \rangle$, for some nonzero vector $c \in \R^d$, so we can use tools from one-dimensional ergodic optimization.

Let us describe a very important example that appeared in many of the early results in ergodic optimization \cite{Conze-G,Hunt-Ott,Jenk.thesis,Jenk.fish,Bousch.poisson}.
Let $T(z) \coloneq z^2$ be the doubling map on the unit circle, and let $f \colon S^1 \to \C$ be the inclusion function. The associated rotation set $R(f) \subset \C = \R^2$ is called the \emph{fish}.
Confirming previous observations from other researchers, Bousch \cite{Bousch.poisson} proved that the extremal measures are exactly the Sturmian measures. These measures form a family $\mu_\rho$ parametrized by rotation number $\rho \in \R/\Z$. If $\rho = p/q$ is rational then $\mu_\rho$ is supported on a periodic orbit of period $q$, while if $\rho$ is irrational then  $\mu_\rho$ is supported on an extremely thin Cantor set (of zero Hausdorff dimension, in particular) where $T$ is semiconjugated to an irrational rotation. 
In particular, Jenkinson's problem stated at the end of \cref{s.Birkhoff_any_dim} has a positive answer if $f$ is a trigonometric polynomial of degree $1$: Sturmian measures not only have zero entropy (i.e., subexponential complexity), but in fact they have linear complexity.

Let us come back to arbitrary $T$ and $f$.
Analogously to \eqref{e.envelope_beta}, we have the following \emph{enveloping property}:
\begin{equation}\label{e.envelope_any_dim}
R(f) = \bigcap_{n \ge 1} \frac{1}{n} \conv \big( f^{(n)}(X) \big) 
\, ,
\end{equation}
and the intersection is also a limit (in both the set-theoretic and the Hausdorff senses).
Using the same trick as in the proof of \eqref{e.dual}, it follows that for every neighborhood $U$ of $R(f)$ there exists $g$ cohomologous to $f$ whose image is contained in $U$.  

The following question arises: when can we find $g$ cohomologous to $f$ taking values in $R(f)$?
As we have already learned in the previous \lcnamecref{s.Birkhoff_one_dim}, to hope for this to be true we need at least some  hyperbolicity and regularity assumptions.
If moreover $d=1$ then the answer of the question becomes positive:
Bousch \cite{Bousch.bilateral} showed that whenever $T$ and $f$ satisfy the assumptions of a Ma\~n\'e Lemma, there exists $g$ cohomologous to $f$ taking values in $R(f)$, i.e.\ such that $\alpha(f) \le g \le \beta(f)$.
What about $d \ge 2$?
Unfortunately the answer is negative. 
The following observation was found by Vincent Delecroix and the author:

\begin{proposition}\label{p.Mane_counterexample}
Let $(X,T)$ and $f$ be as in the definition of the fish.
There exists no $g$ $C^0$-cohomologous to $f$ taking values in the fish.
\end{proposition}

\begin{proof}
For a contradiction, suppose that there exists a continuous function $h \colon S^1 \to \C$ such that $g \coloneq f + h\circ T - h$ takes values in the fish.
For each integer $n \ge 0$, let $z_n \coloneqq e^{2\pi i / 2^{n}}$. 
These points form a homoclinic orbit 
$$
\dots \mapsto z_2 \mapsto z_1 \mapsto z_0 \hookleftarrow \quad \text{with} \quad \lim z_n  = z_0 \, .
$$
We claim that $g(z_n) = 1$ for every $n$.
This leads to a contradiction, because on one hand, 
the series $\sum_{n=1}^\infty (f(z_n)-1)$ is absolutely convergent to a non-zero sum (as the imaginary part is obviously positive), and on the other hand, by telescopic summation and continuity of $h$, the sum should be zero. 

In order to prove the claim, note that $g$ must constant equal to $\int f \dd \mu_\rho$ on the support of each Sturmian measure $\mu_\rho$.
In particular, the compact set $K \coloneqq (g \circ T - g)^{-1}(0)$ contains all those supports. 
Consider the obvious semiconjugacy $\phi$ between the one-sided two-shift and the doubling map $T$, namely the map  which associates to an infinite word $w = b_0 b_1 \dots$ in zeros and ones the complex number $\phi(w) = e^{2\pi i t}$ where $t = 0.b_1b_2\dots$ in binary. Then $\phi( 0^n 1 0^\infty ) = z_n$. On the other hand, for each $k\ge 0$, the periodic infinite word $( 0^n 1 0^k )^\infty$ is Sturmian, and it tends to $ 0^n 1 0^\infty$ as $k \to \infty$. This shows that each $z_n$ is the limit of points that belong to supports of Sturmian measures. In particular, each $z_n$ belongs to $K$. This means that the value $g(z_n)$ is independent of $n$. Since $z_0 = 1$ is a fixed point, we conclude that this value is $1$, as claimed. 
\end{proof}

Following \cite{Bousch-Jenk}, a function $f$ is called a \emph{weak coboundary} if $\int f \dd \mu = 0$ for every $\mu \in \cM_T$, or equivalently if $f$ is a uniform limit of coboundaries. There exist weak coboundaries $f$ that are not coboundaries; indeed this happens whenever $T$ is a non-periodic homeomorphism: see \cite{Kocsard}. The paper \cite{Bousch-Jenk} contains an explicit example of such $f$ in the case $T$ is the doubling map. The reason why their function $f$ is not a coboundary is that the sum over a homoclinic orbit is nonzero. This is exactly the obstruction we used in the proof of \cref{p.Mane_counterexample}. Therefore we pose the problem: Does it exist a function $g$ weakly cohomologous to $f$ taking values in the fish?

Naturally, there are many other questions about rotation sets.
We can ask about their shape, either for typical or for all functions with some prescribed regularity. 
It is shown in \cite{K-Wolf} that any compact convex subset of $\R^d$ is the rotation set of some continuous function. 
In the case of the fish, the boundary is not differentiable, and has a dense subset of corners, one at each extremal point corresponding to a measure supported on a periodic orbit; furthermore, and all the curvature of the boundary is concentrated at these corners.
This seems to be the typical situation of rotation sets of regular functions. 
Let us note that the boundary of Schartzmann balls is never differentiable: see \cite{Pol-Sharp} and references therein. 

Another property of the fish is the following: the closure of the union of the supports of the extremal measures has zero topological entropy (actually it has cubic complexity; see \cite[Corol.~18]{Mignosi}). Is this phenomenon typical?

\section{Optimization of the top Lyapunov exponent}\label{s.Lyapunov_top}

We now replace Birkhoff sums by matrix products.
That is, given a continuous map $F \colon X \to \Mat(d,\R)$ taking values into the space of $d \times d$ real matrices, we consider the products
$$
F^{(n)}(x) \coloneq F(T^{n-1} x) \cdots F(Tx) F(x) \, .
$$
The triple $(X,T,F)$ is called a \emph{linear cocycle} of dimension $d$.
It induces a skew-product dynamics on $X \times \R^d$ by $(x,v) \mapsto (Tx, F(x) v)$, whose $n$-th iterate is therefore $(x,v) \mapsto (T^n x, F^{(n)}(x) v)$. More generally, we could replace $X \times \R^d$ by any vector bundle over $X$ and then consider bundle endomorphisms that fiber over $T \colon X \to X$, but for simplicity will refrain from doing so.

As an immediate consequence of the Kingman's subadditive ergodic theorem,
for any $\mu \in \cM_T$ and $\mu$-almost every $x \in X$, the following limit, called the \emph{top Lyapunov exponent} at $x$, exists:
$$
\lambda_1 (F, x) \coloneq \lim_{n \to \infty} \frac{1}{n} \log \| F^{(n)}(x) \| \in [-\infty, +\infty), 
$$
which is clearly independent of the choice of norm on the space of matrices.
Similarly to what we did in \cref{s.Birkhoff_one_dim}, we can either minimize or maximize this number; the corresponding quantities will be denoted by $\alpha(F)$ and $\beta(F)$.
However, this time the maximization and the minimization problems are fundamentally different. 
While $\beta(F)$ is always attained by at least one measure (which will be called \emph{Lyapunov maximizing}), that is not necessarily the case for $\alpha(F)$. Indeed, absence of Lyapunov minimizing measures may be locally generic \cite[Rem.~1.13]{B-Morris}, and may occur even for ``derivative cocycles'' \cite{CLR}.

More generally, one can replace $\log \| F^{(n)} \|$ by any subadditive sequence of continuous functions, or even upper semicontinuous ones, and optimize the corresponding asymptotic average. One show check that a maximizing measure always exists, and that a enveloping property similar to \eqref{e.envelope_beta} holds. See the appendix of the paper \cite{Morris.Mather} for the proofs of these and other basic results on subadditive ergodic maximization. From these general results one can derive immediately those of \cite{Cao,Drag}, for example. 

The maximization of the linear escape rate of a cocycle of isometries also fits in the context of subadditive ergodic optmization. Under a nonpositive curvature assumption, this maximal escape rate satisfies a duality formula resembling \eqref{e.dual}: see \cite{B-Navas}. For some information on the maximal escape rate in the case of isometries of Gromov-hyperbolic spaces, see \cite{Oregon1}.

Returning to linear cocycles, note that if the matrices $F(x)$ are invertible then we can define a skew-product transformation $T_F$ on the compact space $X \times \R\P^{d-1}$ by $(x,[v]) \mapsto (Tx, [F(x) v])$. Then $\beta(F)$ can be seen as the maximal ergodic average of the function $f(x,[v]) \coloneq \log (\|F(x) v\|/\|v\|)$. In this way, maximization of the top Lyapunov exponent can be reduced to commutative ergodic optimization. Note, however, that the space of $T_F$-invariant probability measures depends on $F$ in a complicated way. 

There is a specific setting where optimization of the top Lyapunov exponent has been studied extensively.
An \emph{one-step cocycle} is a linear cocycle $(X,T,F)$ where $(X,T)$ is the full shift (either one- or two-sided) on a finite alphabet, say $\{1,\dots,k\}$, and the matrix $F(x)$ only depends on the zeroth symbol of the sequence $x$. Therefore an one-step cocycle is completely specified by a $k$-tuple of matrices $(A_1,\dots,A_k)$. 
It is possible to consider also compact alphabets, but for simplicity let us stick with finite ones.

The \emph{joint spectral radius} and the \emph{joint spectral subradius} of a tuple of matrices are respectively defined as the numbers $e^{\beta(F)}$ and $e^{\alpha(F)}$, where $(X,T,F)$ is the corresponding one-step cocycle. The joint spectral radius was introduced in 1960 by Rota and Strang \cite{Rota-S}, and it became a popular subject in the 1990's as applications to several areas (wavelets, control theory, combinatorics, etc) were found. The joint  spectral subradius was introduced later by Gurvits \cite{Gurvits}, and has also been the subject of some pure and applied research. See the book \cite{Jungers.book} for more information.

The first examples of one-step cocycles with finite alphabets without a Lyapunov-maximizing measure supported on a periodic orbit were first constructed in dimension $d=2$ by Bousch and Mairesse \cite{Bousch-Mai}, refuting the finiteness conjecture from \cite{Lagarias-W}.
Other constructions appeared later: see \cite{BTV}.
Counterexamples to the finiteness conjecture seem to be very non-typical: Maesumi \cite{Maesumi} conjectures that they have zero Lebesgue measure in the space of tuples of matrices, and Morris and Sidorov \cite{Morris-Sidorov} exhibit one-parameter families of pairs of matrices where counterexamples form a Cantor set of zero Hausdorff dimension.

The Lyapunov maximizing measures for the one-step cocycles in the examples from \cite{Bousch-Mai,Morris-Sidorov} (among others)  
are Sturmian and so have linear complexity. 
There are higher-dimensional examples with arbitrary polynomial complexity \cite{HMS}. In all known examples where the Lyapunov maximizing measure is unique, it has subexponential complexity, i.e., zero entropy. So the following question becomes inevitable: is this always the case?

A partial result in that direction was obtained by Rams and the author \cite{B-Rams}. 
We exhibit a large class of $2$-dimensional one-step cocycles 
for which the Lyapunov-maximizing and Lyapunov-minimizing measures have zero entropy;
furthermore the class includes counterexamples to the finiteness conjecture.
Our sufficient conditions for zero entropy are simple: existence of strictly invariant families of cones satisfying a non-overlapping condition. 
(In particular, our cocycles admit a dominated splitting; see \cref{s.Lyapunov_all} below for the definition.)
To prove the result, we identify a certain order structure on Lyapunov-optimal orbits (or more precisely on the Oseledets directions associated to those orbits) that leaves no room for positive entropy.

In the setting considered in \cite{B-Rams} (or more generally for cocycles that admit a dominated splitting of index $1$), Lyapunov-minimizing measures do exist, and moreover the minimal top Lyapunov exponent $\alpha$ is continuous among such cocycles. For one-step cocycles that admit no such splitting, $\beta$ is still continuous, but $\alpha$ is not: see \cite{B-Morris}. In fact,  even though discontinuities of $\alpha$ are topologically non-generic, we believe that they form a set of positive Lebesgue measure: \cite[Conj.~7.7]{B-Morris}.

Let us now come back to general linear cocycles, but let us focus the discussion on the maximal top Lyapunov exponent $\beta$.
As we have seen in \cref{s.Birkhoff_one_dim}, Ma\~n\'e Lemma is a basic tool in $1$-dimensional commutative ergodic optimization. 
Let us describe related a notion in the setting of Lyapunov exponents.

Let $(X,T,F)$ be a linear cocycle of dimension $d$. A \emph{Finsler norm} is a family $\{\| \mathord{\cdot} \|_x\}$ of norms in $\R^d$ depending continuously on $x \in X$. An \emph{extremal norm} is a Finsler norm such that
$$
\|F(x) v \|_{Tx} \le e^{\beta(F)} \|v\|_x \quad \text{for all $x \in X$ and $v \in \R^d$.}
$$
 
In the case $d=1$, the linear maps $F(x) \colon \R \to \R$ consist of multiplication by scalars $\pm e^{f(x)}$, and the maximal Lyapunov exponent $\beta(F)$ of the cocycle equals the maximal ergodic average $\beta(f)$ of the function $f$. 
Moreover, an arbitrary Finsler norm can be written as $\|v\|_x = e^{h(x)} |v|$, and it will be an extremal norm if and only if $f + h \circ T - h \le \beta(f)$, which is the cohomological inequality \eqref{e.Mane}.
So the relation with Ma\~n\'e Lemma becomes apparent and we see that existence of an extremal norm is far from automatic even in dimension $d=1$.

Extremal norms were first constructed by Barabanov \cite{Barabanov} in the case of one-step cocycles: he showed that under an irreducibility assumption (no common invariant subspace, except for the trivial ones), there exists an extremal norm that is constant (i.e.\  $\| v \|_x$ is independent of the basepoint $x$).\footnote{Barabanov's norms also have an extra property that does not concern us here. Previously, Rota and Strang \cite{Rota-S} have already considered the weaker notion of extremal \emph{operator} norms.} 
These extremal norms provide a fundamental tool in the study of the joint spectral radius: see also \cite{Wirth,Jungers.book}.

Beyond one-step cocycles, when can we guarantee the existence of an extremal norm?
Garibaldi and the author \cite{B-Garibaldi} consider the situation where $T$ is a hyperbolic homeomorphism and $F \colon X \to \GL(d,\R)$ is a $\theta$-H\"older continuous map. 
As it happens often in this context, it is useful to assume \emph{fiber bunching}, which roughly means that the largest rate under which the matrices $F(x)$ distort angles is bounded by $\tau^\theta$, where $\tau>1$ is a constant related to the hyperbolicity of $T$.
(Note that locally constant cocycles, being locally constant, are $\theta$-H\"older for arbitrarily large $\theta$, and so always satisfy fiber bunching; in this sense, our setting generalizes the classical one.)
We say that the cocycle is \emph{irreducible} if has no $\theta$-H\"older invariant subbundles, except for the trivial ones. The main result of \cite{B-Garibaldi} is that strong fiber bunching together with irreducibility implies the existence of an extremal norm. 
Let us also mention a curious fact: there are examples where the extremal norm cannot be Riemannian.

The existence of an extremal norm is a first step towards more refined study of maximizing measures: for example, it implies the existence of a \emph{Lyapunov maximizing set}, similarly to the maximizing sets discussed in \cref{s.Birkhoff_one_dim}. Such sets were studied in \cite{Morris.Mather} for one-step cocycles.

We can recast in the present context the same type of questions discussed above: How complex are Lyapunov-maximizing measures, either for typical cocycles, or (assuming uniqueness) among all cocycles within a prescribed regularity class?

\section{Optimization of all Lyapunov exponents}\label{s.Lyapunov_all}

In this final \lcnamecref{s.Lyapunov_all}, we consider all Lyapunov exponents and not only the top one. 
Our aim is modest: to introduce an appropriate setting for ergodic optimization of Lyapunov exponents, and to check that the most basic properties seen in the previous sections are still valid.

Let $\s_1(g) \ge \dots \ge \s_d(g)$ denote the \emph{singular values} of a matrix $g \in \GL(d,\R)$. 
These are the semi-axes of the ellipsoid $g(S^{d-1})$, where $S^{d-1}$ is the unit sphere in $\R^d$.
The \emph{Cartan projection} is the map
\begin{equation}\label{e.Cartan}
\vec{\sigma}(g) \coloneq \big( \log \s_1(g), \dots, \log \s_d(g) \big) \, ,
\end{equation}
which takes values in the \emph{positive chamber}
$$
\fa^+ \coloneq \big\{(\xi_1,\dots,\xi_d) \in \R^d \st \xi_1\ge \cdots \ge \xi_d\big\} \, .
$$
The Cartan projection has the subadditive property 
\begin{equation}\label{e.Cartan_subadd}
\vec{\sigma}(gh) \preccurlyeq \vec{\sigma}(g) + \vec{\sigma}(h) \, ,
\end{equation} 
where $\preccurlyeq$ denotes the \emph{majorization partial order} in $\R^d$ defined in as follows: 
$\xi \preccurlyeq \eta$ (which reads as $\xi$ is \emph{majorized} by $\eta$)
if $\xi$ is a convex combination of vectors obtained by permutation of the entries of $\eta$.
The group of automorphisms of $\R^d$ consisting of permutation of coordinates is called the \emph{Weyl group} and is denoted $\cW$; so $\xi \preccurlyeq \eta$ if and only if $\xi$ belongs to the polyhedron $\conv (\cW \eta)$ (called a \emph{permuthohedron}).
For vectors $\xi = (\xi_1,\dots,\xi_d)$ and $\eta=(\eta_1,\dots,\eta_d)$ in the positive chamber $\fa^+$, majorization can be characterized by the following system of inequalities:
$$
\xi \preccurlyeq \eta \quad \Leftrightarrow \quad \forall i \in \{1,\dots,d\}, \
\xi_1 + \cdots + \xi_i \le \eta_1 + \cdots + \eta_i \, , \text{with equality if $i=d$.}
$$
For a proof, see the book \cite{MOA}, which contains plenty of information on majorization, including applications.

Now let us consider a linear cocycle $(X,T,F)$.
For simplicity, let us assume that the matrices $F(x)$ are invertible.
Using Kingman's theorem, one shows that for every $\mu \in \cM_T$ and $\mu$-almost every $x \in X$, the limit
$$
\vec\lambda (F, x) \coloneq \lim_{n \to \infty} \frac{1}{n} \vec\sigma(F^{(n)}(x)) \, 
$$
exists; it is called the \emph{Lyapunov vector} of the point $x$.
Its entries are called the \emph{Lyapunov exponents}. 
If $\mu$ is ergodic then the Lyapunov vector is $\mu$-almost surely equal to a constant $\vec\lambda (F, \mu)$.
The \emph{Lyapunov spectrum} of the cocycle is defined as:
$$
L^+(F) \coloneq \left\{ \vec\lambda (F, \mu) \st \mu \in \cE_T \right\} \subset \fa^+ \, .
$$
By analogy with the rotation set \eqref{e.rot_set_erg}, we introduce the \emph{inner envelope} of the cocycle as
the closed-convex hull of the Lyapunov spectrum, that is, 
$I^+(F) \coloneq \overline{\conv} (L^+(F))$. \footnote{In recent work, Sert \cite{Sert} considers one-step cocycles taking values on more general Lie groups and satisfying a Zariski denseness assumption, introduces and studies a subset of the positive chamber called \emph{joint spectrum}, and applies it to obtain results on large deviations. It turns out that the joint spectrum coincides with our inner envelope $I^+(F)$, in the $\SL(d,\R)$ case at least.}
Differently from the commutative situation, however, extremal points of this convex set are not necessarily attained.
Therefore we introduce other sets (see \cref{f.spectra}):
\begin{alignat*}{2}
L(F)	&\coloneq \cW \cdot L^+(F)	&\quad &\text{\emph{(symmetric Lyapunov spectrum)};} \\
I(F)	&\coloneq \cW \cdot I^+(F)	&\quad &\text{\emph{(symmetric inner envelope)};} \\
O(F)	&\coloneq \conv (I(F))		&\quad &\text{\emph{(symmetric outer envelope)};} \\
O^+(F)	&\coloneq O(F) \cap \fa^+	&\quad &\text{\emph{(outer envelope)}.} 
\end{alignat*}
Then the extremal points of the symmetric outer envelope $O(F)$ 
are attained as (perhaps reordered) Lyapunov vectors of ergodic measures:
\begin{equation}\label{e.ext_outer}
\ext(O(F)) \subseteq L(f)  \, .
\end{equation}
Sometimes this is the best we can say about attainability, but sometimes we can do better.
There is one situation where all extremal points of the symmetric inner envelope $I(f)$ are attained, namely if the cocycle admits a dominated splitting into one-dimensional bundles, because then we are essentially reduced to rotation sets in $\R^d$.

Let us explain the concept of domination.
For simplicity, let us assume that $T \colon X \to X$ is a homeomorphism. Let us also assume that there is a fully supported $T$-invariant probability measure (otherwise we simply restrict $T$ to the minimal center of attraction; see \cite[Prop.\ 8.8(c)]{Akin}).

Suppose that $V$ and $W$ are two $F$-invariant subbundles of constant dimensions.
We say that \emph{$V$ dominates $W$} if there are constants $\kappa_0 > 1$ and $n_0 \ge 1$ such that 
for every $x \in X$ and every $n \ge n_0$, the smallest singular value of $F^{(n)}(x)|_{V(x)}$ is bigger than $\kappa_0^n$ times the biggest singular value of $F^{(n)}(x)|_{W(x)}$. We say that $V$ is the \emph{dominating} bundle, and $W$ is the \emph{dominated} bundle. The terminology \emph{exponentially separated splitting} is more common in ODE and Control Theory, and other terms also appear especially in the earlier literature, but we will stick to the terminology \emph{dominated splitting}, though grammatically inferior. The bundles $V$ and $W$ are in fact continuous, and they are robust with respect to perturbations of $F$: see \cite{BDV.book} for this and other basic properties of domination.

The cocycle has a unique \emph{finest dominated splitting}; this is a finite collection of invariant subbundles $V_1$, $V_2$, \dots, $V_k$, each one dominating the next one, and with maximal $k$. It is indeed a splitting in the sense that $V_1(x) \oplus \cdots \oplus V_k(x) = \R^d$ for every $x$.
If $k=1$ then the splitting is called \emph{trivial}.

We say that $i \in \{1,\dots,d-1\}$ is a \emph{index of domination} of the cocycle if there exists a dominated splitting with a dominating bundle of dimension $i$; otherwise we say that that $i$ is a \emph{index of non-domination}. 
So, if $V_1 \oplus \cdots \oplus V_k$ is the finest dominated splitting of the cocycle, and $d_i \coloneq \dim V_i$, then the indices of domination are $d_1$, $d_1+d_2$, \dots, $d_1 + d_2 + \cdots + d_{k-1}$. 

There is a way of detecting domination without referring to invariant subbundles or cones. 
As shown in \cite{B-Gourmelon},
$i \in \{1,\dots,d-1\}$ is a index of domination if and only if there is an exponential gap between $i$-th and $(i+1)$-th singular values; more precisely:
there are constants $\kappa_1>1$ and $n_1 \ge 1$ such that $\s_i(F^{(n)}(x))/\s_{i+1}(F^{(n)}(x)) \ge \kappa_1^n$ for all $x \in X$ and all $n \ge n_1$.
In terms of the sets
\begin{equation}\label{e.Sigma}
\Sigma_n(F) \coloneq \big\{ \vec{\sigma}(F^{(n)}(x)) \st x \in X \big\} \subset \fa^+ \, ,
\end{equation}
we have the following geometric characterization: $i$ is an index of domination if and only if for all sufficiently large $n$, the sets $\frac{1}{n} \Sigma_n(F)$ are uniformly away from the \emph{wall} $\xi_i = \xi_{i+1}$ (a hyperplane that contains part of the boundary of the positive chamber $\fa^+$).

If $\Theta$ is a subset of $\{1,2,\dots,d-1\}$, define the \emph{$\Theta$-superchamber} as the following closed convex subset of $\R^d$:
$$
\fa^\Theta \coloneq \big\{ 
(\xi_1,\dots,\xi_d) \in \R^d \st 1 \le i \le j < k \le d \text{ integers, } j \not\in\Theta \ \Rightarrow \ \xi_i \ge \xi_k \big\} \, . 
$$
For example, $\fa^\Theta = \fa^+$ if $\Theta$ is empty, and $\fa^\Theta =  \R^d$ if $\Theta = \{1,2,\dots,d-1\}$.
(Moreover, $\fa^\Theta$ equals the orbit of $\fa^+$ under an appropriate subgroup of $\cW$, but we will not need this fact.)
If $C$ is any subset of $\R^d$, the \emph{closed-$\Theta$-convex} hull of $C$, denoted by $\overline{\conv}_\Theta(C)$, is defined as the smallest closed subset of $\R^d$ that contains $C$, is invariant under the Weyl group $\cW$, and whose intersection with the superchamber $\fa^\Theta$ is convex.

Let $\Theta$ be the set of indices of non-domination of the cocycle $(X,T,F)$.
We define the following two sets:
\begin{alignat*}{2}
M(F) 	&\coloneq \overline{\conv}_\Theta \left( L^+(F) \right)	&\quad &\text{\emph{(symmetric Morse spectrum)};} \\
M^+(F) 	&\coloneq \fa^+ \cap M(F) 								&\quad &\text{\emph{(Morse spectrum)}.}
\end{alignat*}
The Morse spectra (symmetric or not) are sandwiched between the inner and outer envelopes: see \cref{f.spectra}).
If $\Theta = \emptyset$ then $M(F) = I(F)$, while if $\Theta = \{1,\dots,d-1\}$ then $M(F) = O(F)$.


\begin{figure}[hbt]
\begin{center}
\begin{tikzpicture}[scale=.23,font=\footnotesize]
	
	\draw[fill=gray!20] (9,6) arc(0:60:2) --(2.69,10.79) arc(60:120:2) --(-9.7,4.79) arc(120:180:2) --(-10.7,-3.06) arc(180:240:2) --(.7,-10.79) arc(240:300:2) --(8,-7.73) arc(300:360:2) --cycle;

	\draw[fill=gray!60] (6,4.27) arc(-120:60:2) --(2.69,10.79) arc(60:240:2) --cycle;  
	\draw[fill=gray!60] (6,-4.27) arc(120:-60:2) --(2.69,-10.79) arc(-60:-240:2) --cycle;  
	\draw[fill=gray!60] (-6.7,3.06) arc(0:180:2) --(-10.7,-3.06) arc(180:360:2) --cycle;
	
	\fill[gray] (7,6)		circle(2);
	\fill[gray]  (1.7,9.06)	circle(2);
	\fill[gray]  (-8.7,3.06)circle(2);
	\fill[gray] (-8.7,-3.06)circle(2);
	\fill[gray]  (1.7,-9.06)circle(2);
	\fill[gray]  (7,-6)		circle(2);

	\draw (6,4.27) arc(-120:60:2) --(2.69,10.79) arc(60:240:2) --cycle;  
	\draw (6,-4.27) arc(120:-60:2) --(2.69,-10.79) arc(-60:-240:2) --cycle;  
	\draw (-6.7,3.06) arc(0:180:2) --(-10.7,-3.06) arc(180:360:2) --cycle;
	
	\draw[thick,dashed]	(8,7.73)		arc(60:240:2);
	\draw[thick] 		(8,7.73)		arc(60:-120:2);
	\draw[thick]		(2.69,10.79)	arc(60:240:2);
	\draw[thick,dashed] (2.69,10.79)	arc(60:-120:2);
	\draw[thick,dashed]	(8,-7.73)		arc(-60:-240:2);
	\draw[thick] 		(8,-7.73)		arc(-60:120:2);
	\draw[thick]		(2.69,-10.79)	arc(-60:-240:2);
	\draw[thick,dashed] (2.69,-10.79)	arc(-60:120:2);
	\draw[thick]		(-6.7,3.06)		arc(0:180:2);
	\draw[thick,dashed]	(-6.7,3.06)		arc(0:-180:2);
	\draw[thick,dashed]	(-6.7,-3.06)	arc(0:180:2);
	\draw[thick]		(-6.7,-3.06)	arc(0:-180:2);

	\draw (-12,0)--(12,0)			node[right]			{wall $\xi_1=\xi_2$};
	\draw (-6,-10.39)--(6,10.39)	node[above right]	{wall $\xi_2=\xi_3$};
	\draw (6,-10.39)--(-6,10.39)	node[above left]	{wall $\xi_1=\xi_3$};
	
	\draw[white] (7,6) node {\small $L^+(F)$};
	\draw (11,6) node[right] {positive chamber};
	
\end{tikzpicture}
\end{center}
\caption{Suppose $F$ takes values in $\SL(3,\R)$; then all spectra are  contained in the plane $ \{(\xi_1,\xi_2,\xi_3) \in \R^3 \st \xi_1+\xi_2+\xi_3=0\}$. The figure shows a possibility for the three sets $I(F) \subseteq M(F) \subseteq O(F)$, which are are pictured in decreasing shades of gray, assuming that the unique index of domination is $1$, i.e., $\Theta = \{2\}$.}
\label{f.spectra}
\end{figure}
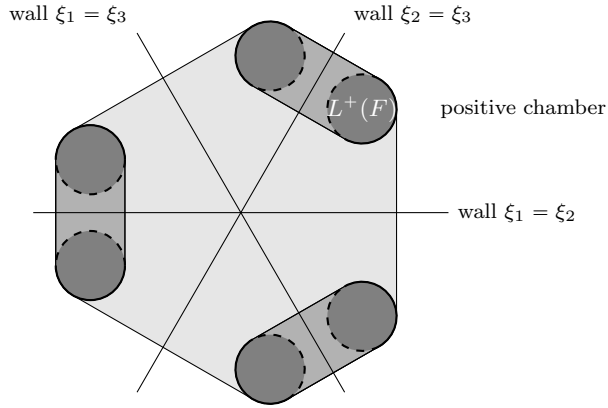

The Morse spectrum allows us to state the following \emph{attainability property}, which is stronger than \eqref{e.ext_outer}:
\begin{equation}\label{e.ext_Morse}
\ext(M(F)) \subseteq L(F) \, . 
\end{equation}

The name ``Morse'' comes from Morse decompositions in Conley theory \cite{Conley}.
Morse spectra were originally defined by Colonius and Kliemann \cite{CK.morse}; see also \cite{CFJ}.
The Morse spectra defined here are more closely related to the ones considered by San~Martin and Seco \cite{San-Seco} (who in fact dealt with more general Lie groups).
In concrete terms, we have the following characterization: 
$\xi \in M^+(F)$ if and only if there exist
sequences $n_i \to \infty$ and $\epsilon_i \to 0$,
$\epsilon_i$-pseudo orbits $(x_{i,0}, x_{i,1}, \dots)$,
and matrices $g_{i,j} \in \GL(d,\R)$ with $\|g_{i,j} - F(x_{i,j})\| < \epsilon_i$ such that: 
$$
\frac{1}{n_i} \vec{\sigma} \big( g_{i, n_i-1} \cdots g_{i,1} g_{i,0} \big) \to \xi \, .
$$
We will not provide a proof.
For background on Morse spectra defined in terms of pseudo orbits and relations with dominated splittings and Lyapunov exponents (and without Lie algebra terminology), see the book \cite{CK.book2}.
Let us remark that all those type of Morse spectra contain more information that the Mather \cite{Mather.spectrum} and Sacker--Sell spectra \cite{Drag}. Indeed, the finest dominated splitting refines the Sacker--Sell decomposition, which may be seen as the finest \emph{absolute} (as opposed to \emph{pointwise}) dominated splitting.

Let us inspect other basic properties of the Morse spectra which are similar to properties of the rotation sets seen in \cref{s.Birkhoff_any_dim}. The following \emph{enveloping property} is analogous to \eqref{e.envelope_any_dim}:
\begin{equation}\label{e.envelope_Morse}
M(F) = \bigcap_{n \in \N} \frac{1}{n} \overline{\conv}_\Theta( \Sigma_n(F) )   \,  ,
\end{equation}
and again the intersection above is a limit.
In particular, the Morse spectrum $M(\mathord{\cdot})$ is an upper semicontinuous function of $F$ (with respect to the uniform topology). If the cocycle has a dominated splitting into $k=d$ bundles then $M(\mathord{\cdot})$ is continuous at $F$. 

Two cocycles $F$ and $G$ over the same base $(X,T)$ are called \emph{conjugate} if there is continuous map $H \colon X \to \GL(d,\R)$ such that $G(x) =  H(Tx)^{-1} F(x) H(x)$. 
The Morse spectrum is invariant under cocycle conjugation.

Let us state a result similar to the duality property \eqref{e.dual} and its vectorial counterpart explained in \cref{s.Birkhoff_any_dim}:

\begin{proposition}\label{p.approx_Mane}
Given a neighborhood $U$ of $M^+(F)$, there exists a cocycle $G$ conjugate to $F$ such that 
$\Sigma_1(G) \subset U$.
\end{proposition}

The proof requires some preliminaries.
Let $\cS$ denote the space of inner products in $\R^d$.
The group $\GL(d,\R)$ acts transitively on $\cS$ as follows:
$$
\biangle{\mathord{\cdot}, \mathord{\cdot}}_2 = g * \biangle{\mathord{\cdot}, \mathord{\cdot}}_1 
\quad \Leftrightarrow \quad
\biangle{u,v}_2 = \biangle{g^{-1} u , g^{-1} v}_1  \, .
$$
Using the standard inner product $\langle \mathord{\cdot}, \mathord{\cdot} \rangle \eqcolon o$ as a reference, every element $\biangle{\mathord{\cdot}, \mathord{\cdot}}$ of $\cS$ can be uniquely represented by a positive (i.e.\ positive-definite symmetric) matrix $p$ such that $\biangle{u,v} = \langle p^{-1} u, v \rangle$.
In this way we may identify $\cS$ with the set of positive matrices, and $o$ is identified with the identity matrix. 
In these terms, the group action becomes: 
$$
g * p = g p g^\mathsf{t} \, .
$$

The \emph{vectorial distance} is defined as the following map:
$$
\vec{\delta} \colon \cS \times \cS \to \fa^+ \, , \qquad
\vec{\delta}(p,q) \coloneq 2 \vec{\sigma}(p^{-1/2}q^{1/2}) \, ,
$$
where $\vec{\sigma}$ is the Cartan projection \eqref{e.Cartan}.
It has the following properties:
\begin{enumerate}
\item $\vec{\delta} (o, q) = \vec{\sigma}(q)$;
\item $\vec{\delta}$ is a complete invariant for the group action on pairs of points, that is, $\vec{\delta} (p_1, q_1) = \vec{\delta} (p_2, q_2)$ if and only if there exists $g \in \GL(d, \R)$ such that $g*p_1=p_2$ and $g*q_1=q_2$;
\item $\vec{\delta}(p,p) = 0$;
\item $\vec{\delta}(q,p) = \mathbf{i} (\vec{\delta}(p,q) )$, where 
$\mathbf{i} (\xi_1,\dots,\xi_d) \coloneqq (-\xi_d,\dots,-\xi_1)$
is the \emph{opposition involution};
\item triangle inequality: $\vec{\delta}(p,r) \preccurlyeq \vec{\delta}(p,q) + \vec{\delta}(q,r)$; this follows from \eqref{e.Cartan_subadd}.
\end{enumerate}
In particular, the euclidian norm of $\vec{\delta}$ is a true distance function, and it invariant under the action; indeed that is the usual way to metrize $\cS$.

A parameterized curve $\gamma \colon [0,1] \to \cS$ is called a \emph{geodesic segment} if there is a vector $\xi \in \fa^+$ such that $\vec{\delta}(f(t), f(s)) = (s-t) \xi$, provided $t \le s$. 
A geodesic segment is determined uniquely by its endpoints $p = f(0)$ and $q =f(1)$; it is given by the formula $f(s) = q^s$ if $p=o$. The image of $f$ is denoted $[p,q]$ and by abuse of terminology is also called a geodesic segment. The \emph{midpoint} of the geodesic segment is $\midpoint [p,q] \coloneqq f(1/2)$.

We shall prove the following vectorial version of the Busemann nonpositive curvature inequality: 
\begin{equation}\label{e.Bus}
\vec{\delta} \big( \midpoint[r,p] , \midpoint[r,q] \big) \preccurlyeq \tfrac{1}{2} \vec{\delta}(p , q) \, , \quad
\text{for all } r, p, q \in \cS \, .
\end{equation}
Parreau \cite{Parreau} has announced a general version of this inequality that holds in other symmetric spaces and affine buildings, using the appropriate partial order.

In order to prove \eqref{e.Bus}, consider the \emph{Jordan projection} $\vec{\chi} \colon \GL(d,\R) \to \fa^+$ defined by 
$\vec{\chi}(g) \coloneq \big( \log |z_1|, \dots, \log |z_d| \big)$, 
where $z_1$, \dots, $z_d$ are the eigenvalues of $g$, ordered so that $|z_1| \ge \cdots \ge |z_d|$.
The Jordan projection is cyclically invariant, that is, $\vec{\chi}(gh) = \vec{\chi}(hg)$.
The Cartan and Jordan projections are related by $\vec{\sigma}(g) = \frac{1}{2} \vec{\chi}(gg^\mathsf{t})$.
Another property is that Cartan majorizes Jordan: $\vec{\sigma}(g) \succcurlyeq \vec{\chi} (g)$.
This follows from the fact that the spectral radius of a matrix is less than or equal to its top singular value, applied to $g$ and its exterior powers. 

\begin{proof}[Proof of the vectorial Busemann NPC inequality \eqref{e.Bus}]
Take arbitrary $r$, $p$, $q \in \cS$.
Since the vectorial distance is invariant under the group action, it is sufficient to consider the case where $r = o$.
Then the midpoints under consideration are $p^{1/2}$ and $q^{1/2}$.
Using the definition of $\vec{\delta}$ and the properties of the projections $\vec{\sigma}$ and $\vec{\chi}$, we have:
\begin{multline*}
\vec{\delta} (p^{1/2} , q^{1/2}) 
= 2 \vec{\sigma} \big(p^{-1/4} q^{1/4}\big) 
= \vec{\chi} \big(p^{-1/4} q^{1/2} p^{-1/4}\big) \\ 
= \vec{\chi} \big(p^{-1/2} q^{1/2}\big)
\preccurlyeq \vec{\sigma} \big(p^{-1/2} q^{1/2}\big) 
= \tfrac{1}{2} \vec{\delta} (p, q) \, . \qedhere
\end{multline*}
\end{proof}

\begin{proof}[Proof of \cref{p.approx_Mane}]
We will adapt the argument from \cite[p.~383--384]{B-Navas}, using \eqref{e.Bus} instead of the ordinary Busemann NPC inequality.
Consider the case $\Theta = \{1,\dots,d-1\}$, i.e., the cocycle no nontrivial dominated splitting.
Then $M^+(F)$ is closed under majorization, in the sense that:
$$
\xi \in \fa^+, \ \eta \in M^+(F), \ \xi \preccurlyeq \eta \quad \Rightarrow \quad \xi \in M^+(F)
$$
Fix a neighborhood $U \supset M^+(F)$.
Without loss of generality, we may assume that $U \cap \fa^+$ is closed under majorization.
We want to find $G$ conjugate to $F$ such that $\Sigma_1(G) \subset U$.
By \eqref{e.envelope_Morse}, for sufficiently large $N$ we have $\frac{1}{N} \Sigma_N(F) \subset U$. 
Fix such $N$ of the form $N = 2^k$.

Let us define recursively continuous maps $\psi_0$, $\psi_1$, \dots, $\psi_k \colon X \to \cS$ as follows:
$\psi_{0}$ is constant equal to $o$, and 
$$
\psi_{j+1}(x) \coloneqq 
\midpoint \big[ (F^{(2^{k-j-1})}(x))^{-1} * \psi_{j}(T^{2^{k-j-1}} x), \psi_{j}(x) \big] \, ,
$$
Then:
$$
\underbrace{\psi_{j+1}(T^{2^{k-j-1}} x)}_{\circled{1}} = 
\midpoint \big[ \underbrace{(F^{(2^{k-j-1})}(T^{2^{k-j-1}} x))^{-1} * \psi_{j}(T^{2^{k-j}} x)}_{\circled{2}}, \psi_{j}(T^{2^{k-j-1}} x) \big] 
$$
and, using equivariance of midpoints,
$$
\underbrace{F^{(2^{k-j-1})}(x) * \psi_{j+1}(x)}_{\circled{3}} =
\midpoint \big[ \psi_{j}(T^{2^{k-j-1}} x), \underbrace{F^{(2^{k-j-1})}(x) * \psi_{j}(x)}_{\circled{4}} \big] \, .
$$
By \eqref{e.Bus}, we have
$\vec{\delta} \big( \circled{1}, \circled{3} \big) \preccurlyeq \tfrac{1}{2}
\vec{\delta} \big( \circled{2}, \circled{4} \big)$,
which, by the invariance of the vectorial distance, amounts to
$$
\vec{\delta} \big(\psi_{j+1}(T^{2^{k-j-1}} x), F^{(2^{k-j-1})}(x) * \psi_{j+1}(x) \big)
\preccurlyeq \tfrac{1}{2}
\vec{\delta} \big( \psi_{j}(T^{2^{k-j}} x), F^{(2^{k-j})}(x) * \psi_{j}(x) \big) \, .
$$
Combining the whole chain of these inequalities we obtain:
$$
\vec{\delta} \big( \psi_{k}(T x), F(x) * \psi_{k}(x) \big)
\preccurlyeq \tfrac{1}{2^k}
\vec{\delta} \big( \psi_{0}(T^{2^k} x), F^{(2^k)}(x) * \psi_{0}(x) \big)
\, .
$$
Equivalently, denoting $\phi \coloneqq \psi_k$,
\begin{equation}\label{e.oba}
\vec{\delta} \big( \phi(T x), F(x) * \phi(x) \big)
\preccurlyeq 
\tfrac{1}{N} \vec{\delta} \big( o,  F^{(N)}(x) * o \big) 
\, .
\end{equation}
Take a continuous map $H \colon X \to \GL(d,\R)$ such that $H(x) * \phi(x) = o$ for every $x$ (e.g., $H \coloneqq \phi^{-1/2}$),
and let $G(x) \coloneqq H(T x) F(x) H(x)^{-1}$.
Then 
\begin{multline*}
\vec{\delta} \big( \phi(T x), F(x) * \phi(x) \big) = \vec{\delta} \big( o, G(x) * o \big)
\\ = 
\vec{\sigma} \big( G(x) * o \big)
\preccurlyeq 
\tfrac{1}{N} \vec{\sigma} \big( F^{(N)}(x) * o \big) \in \tfrac{1}{N} \Sigma_N(F) \subseteq U \cap \fa^+ \, .
\end{multline*}
Since $U \cap \fa^+$ is closed under majorization, we conclude that $\vec{\sigma} \big( G(x) * o \big) \in U$.
That is, $\Sigma_1(G) \subseteq U$, as we wanted to show.

In the case the cocycle admits a nontrivial dominated splitting, we take a preliminary conjugation to make the bundles of the finest dominated splitting orthogonal. Then the exact same procedure above leads to the desired conjugation, but we omit the verifications.
\end{proof}

As a corollary of \cref{p.approx_Mane}, we reobtain a result of Gourmelon \cite{Gourmelon}, which says that it is always possible to find an adapted Riemannian norm for which dominations are seen in the first iterate (i.e., $n_0=1$ in our definition). Indeed, the corresponding inner product at the point $x$ is $\phi(x)$, where $\phi$ is the map constructed in the proof of \cref{p.approx_Mane}.

Furthermore, the construction gives as extra property which is essential to certain applications \cite{B-K-RH}, namely: fixed a favored ergodic measure $\mu_0 \in \cE_T$, we can choose the adapted metric $\phi$ with respect to which the expansion rates in the first iterate are close to the Lyapunov exponents with respect to $\mu_0$, except on a set of small $\mu_0$ measure.\footnote{Lyapunov metrics used in Pesin theory \cite[p.~668]{KH.book} satisfy such a property in a set of full measure, but they are only measurable, and are not necessarily adapted to the finest dominated splitting.} More precisely, we can take $N$ large enough so that the RHS in \eqref{e.oba} is $L^1(\mu_0)$-close to the Lyapunov vector $\vec{\lambda}(F,\mu_0)$. On the other hand, the integral of the LHS majorizes the Lyapunov vector. It follows that the RHS is also $L^1(\mu_0)$-close to the Lyapunov vector.

The measures $\mu \in \cM_T$ for which the Lyapunov vector $\vec{\lambda}(F,\mu)$ is an extremal point of the symmetric Morse set $M(F)$ are called \emph{extremal measures} for the linear cocycle $(X,T,F)$.

As an example, consider the one-step cocycle generated by the pair of matrices $A_1 \coloneqq \left(\begin{smallmatrix} 1 & 1 \\ 0 & 1\end{smallmatrix}\right)$ and $A_2 \coloneqq \left(\begin{smallmatrix} 2 & 0 \\ 2 & 2 \end{smallmatrix}\right)$. The extremal measures of this cocycle are Sturmian: this can be deduced from a result of \cite{HMST}\footnote{Namely: for every $\alpha>0$, the one-step cocycle generated by $(A_0, \alpha A_1)$ has a unique Lyapunov-maximizing measure, which is Sturmian.}. Moreover, results of \cite{Morris-Sidorov} imply that the boundary of the symmetric Morse set is not differentiable, with a dense subset of corners, just like the fish seen in \cref{s.Birkhoff_any_dim}.
Again, we ask: are these phenomena typical?

\begin{ack}
I thank Ian D.~Morris for corrections and suggestions.
\end{ack}



%
%
%
%
%

\end{document}